\documentclass[12pt]{amsart}
\usepackage{ amsmath, amsthm, amsfonts, amssymb, color}
 \usepackage{mathrsfs}
\usepackage{amsfonts, amsmath}
 \usepackage{amsmath,amstext,amsthm,amssymb,amsxtra}
 \usepackage{txfonts} 
 \usepackage[colorlinks, citecolor=blue,pagebackref,hypertexnames=false]{hyperref}
 \allowdisplaybreaks
 \usepackage{pgf,tikz}

\begin{document}

\hfuzz=6pt

\widowpenalty=10000

\newtheorem{cl}{Claim}
\newtheorem{theorem}{Theorem}[section]
\newtheorem{proposition}[theorem]{Proposition}
\newtheorem{coro}[theorem]{Corollary}
\newtheorem{lemma}[theorem]{Lemma}
\newtheorem{definition}[theorem]{Definition}
\newtheorem{assum}{Assumption}[section]
\newtheorem{example}[theorem]{Example}
\newtheorem{remark}[theorem]{Remark}
\renewcommand{\theequation}
{\thesection.\arabic{equation}}

\def\SL{\sqrt H}

\newcommand{\mar}[1]{{\marginpar{\sffamily{\scriptsize
        #1}}}}

\newcommand{\as}[1]{{\mar{AS:#1}}}

\newcommand\R{\mathbb{R}}
\newcommand\RR{\mathbb{R}}
\newcommand\CC{\mathbb{C}}
\newcommand\NN{\mathbb{N}}
\newcommand\ZZ{\mathbb{Z}}
\def\RN {\mathbb{R}^n}
\renewcommand\Re{\operatorname{Re}}
\renewcommand\Im{\operatorname{Im}}

\newcommand{\mc}{\mathcal}
\newcommand\D{\mathcal{D}}
\def\hs{\hspace{0.33cm}}
\newcommand{\la}{\alpha}
\def \l {\alpha}
\newcommand{\eps}{\varepsilon}
\newcommand{\pl}{\partial}
\newcommand{\supp}{{\rm supp}{\hspace{.05cm}}}
\newcommand{\x}{\times}
\newcommand{\lag}{\langle}
\newcommand{\rag}{\rangle}

\newcommand\wrt{\,{\rm d}}

\title[Weak type $(p,p)$ bounds for
Schr\"odinger  groups via generalized Gaussian estimates]
{Weak type $(p,p)$ bounds for
Schr\"odinger  groups via generalized Gaussian estimates}
\author{Zhijie Fan}

\address{Zhijie Fan, Department of Mathematics, Sun Yat-sen (Zhongshan)
University, Guangzhou, 510275, P.R. China}
\email{fanzhj3@mail2.sysu.edu.cn}

 \date{\today}

\begin{abstract}
Let $L$ be a non-negative self-adjoint operator acting on $L^2(X)$,
where $X$ is a space of homogeneous type with a dimension $n$. Suppose that
the heat operator $e^{-tL}$
satisfies  the generalized   Gaussian $(p_0, p'_0)$-estimates of order $m$ for some $1\leq p_0 < 2$.
It is known  that the operator
 $(I+L)^{-s } e^{itL}$ is bounded on $L^p(X)$ for $s\geq n|{1/ 2}-{1/p}| $ and $ p\in (p_0, p_0')$
  (see for example, \cite{Blunck2, BDN, CCO, CDLY, DN, Mi1}).
 In this paper we  study the endpoint case $p=p_0$ and show that for $s_0= n\big|{1\over  2}-{1\over  p_0}\big|$,
 the operator $(I+L)^{-{s_0}}e^{itL}  $ is of weak type
 $(p_{0},p_{0})$, that is,
  there is a constant $C>0$, independent of $t$ and $f$ so that
\begin{eqnarray*}
\mu\left(\left\{x: \big|(I+L)^{-s_0}e^{itL} f(x)\big|>\alpha \right\} \right)\leq C
(1+|t|)^{n(1 - {p_0\over 2}) }
\left( {\|f\|_{p_0} \over \alpha} \right)^{p_0} ,  \ \ \ t\in{\mathbb R}
\end{eqnarray*}
 for $\alpha>0$ when $\mu(X)=\infty$, and $\alpha>\big(\|f\|_{p_{0}}/\mu(X) \big)^{p_{0}}$ when $\mu(X)<\infty$.

 Our results can be applied to
 Schr\"odinger operators with rough  potentials and
  higher order elliptic operators with bounded measurable coefficients   although in general, their semigroups fail to satisfy Gaussian upper bounds.

\end{abstract}

\maketitle


\section{Introduction}
Throughout the paper we suppose  that  $(X, d, \mu)$ is a metric measure space  with a distance function $d$
and  a measure  $\mu$.
We say that $(X, d, \mu)$ satisfies
 the doubling property (see Chapter 3, \cite{CW})
if there  exists a constant $C>0$ such that
\begin{eqnarray}\label{e2.1}
\mu(B(x,2r))\leq C \mu(B(x, r))\quad \forall\,r>0,\,x\in X,
\end{eqnarray}
 Note that the doubling property  implies the following
strong homogeneity property,
\begin{align}\label{doubling}
\mu(B(x,\lambda r))\leq C\lambda^{n}\mu(B(x,r)).
\end{align}
 for some $C, n>0$ uniformly for all $\lambda\geq 1$ and $x\in X$.
In Euclidean space with Lebesgue measure, the parameter $n$ corresponds to
the dimension of the space,
  but in our more abstract setting, the optimal $n$
 need not even be an integer.

 Let $L$ be a non-negative self-adjoint operator on the Hilbert space $L^2(X).$
 Consider
 the Schr\"odinger equation in $X\times {\mathbb R},$
 \begin{eqnarray*}\label{e1.00}
\left\{
\begin{array}{ll}
  i{\partial_t u } + L u=0,\\[4pt]
 u|_{t=0}=f
\end{array}
\right.
\end{eqnarray*}
with initial data $f$. Then
the solution can be formally written as
 \begin{equation}\label{e1.1n}
 u(x,t)=e^{itL}f(x) =   \int_0^{\infty}    e^{it\lambda}dE_L(\lambda) f(x), \ \ \ \ t\in{\mathbb R}
   \end{equation}
 for $f\in L^2(X)$, where   $E_L$ denotes the resolution of the identity associated with  $L.$
   By the spectral theorem (\cite{Mc}),  the operator   $  e^{itL}$  is  continuous on $L^2(X)$,  and forms   the Schr\"odinger group.
   A natural problem is to study the mapping properties
of families of operators derived from the Schr\"odinger group
  on various functional spaces defined on $X$.
      This   has attracted a lot of attention in the last   decades,
 and  has been a  very active research topic  in harmonic analysis and partial differential equations-- see for example,
    \cite{A, Blunck2, Br, BDN, CCO, DN,  EK, H,  H1, JN, JN2, La, Lo, Mi1, O, Sj}.

In this paper, we consider a non-negative self-adjoint operator $L$  and numbers
  $m\geq 2$ and $ 1\leq p_0\leq 2 $. Following \cite{Blunck00, Blunck2, CDLY},
we say that   the semigroup $e^{-tL}$  generated by $L$,  satisfies
 the generalized Gaussian  $(p_0, p'_0)$-estimate   of order $m$,
if there exist constants $C, c>0$ such that
\begin{equation*}
 \label{GGE}
\tag{${\rm GGE_{p_0,p'_0, m} }$}
\big\|P_{B(x, t^{1/m})} e^{-tL} P_{B(y, t^{1/m})}\big\|_{p_0\to {p'_0}}\leq
C \mu(B(x,t^{1/m}))^{-({\frac{1}{ p_0}}-{1\over p'_0})} \exp\left(-c\left({d(x,y)^m \over    t }\right)^{1\over m-1}\right)
\end{equation*}
for  every $t>0$ and $x, y\in X$.
Note that condition \eqref{GGE} for the special case $p_0=1$ is equivalent to
$m$-th order   Gaussian estimates (see   for example, \cite{Blunck2}). This means that the
semigroup $e^{-tL}$ has integral kernels $p_t(x,y)$ satisfying the following Gaussian upper estimate:
\begin{equation*}
 \label{GE}
 \tag{${\rm GE}_m$}
|p_t(x,y)| \leq {C\over \mu(B(x,t^{1/m}))} \exp\left(-c \, {  \left({d(x,y)^{m}\over    t}\right)^{1\over m-1}}\right)
\end{equation*}
for every $t>0, x, y\in X$, where $c, C$ are two positive constants and $m\geq 2.$
There are numbers of operators which satisfy generalized Gaussian estimates and, among them,
 there exist many for which classical Gaussian estimates \eqref{GE}  fail.
  This happens, e.g., for Schr\"odinger operators with rough
 potentials \cite{ScV}, second order elliptic operators with rough  lower order terms \cite{LSV}, or
 higher order elliptic operators with bounded measurable coefficients
 \cite{D2}.

 Recently, under the assumption $({\rm GGE_{p_{0},p_{0}^{\prime},m}})$, Chen, Duong, Li and Yan \cite{CDLY} showed that if $L$
satisfies the estimate $({\rm GGE_{p_{0},p_{0}^{\prime},m}})$  for some
  $m\geq 2$ and $1\leq p_{0}< 2$, then for every $p\in (p_0, p'_0)$, there exists a  constant $C=C(n,p)>0$ independent of $t$ and $f$ such that
\begin{eqnarray}\label{Lpbound}
 \left\| (I+L)^{-s }e^{itL} f\right\|_{L^{p}(X)} \leq C (1+|t|)^{s} \|f\|_{L^{p}(X)}, \ \ \ t\in{\mathbb R}, \ \
  \ s\geq n\left|{1\over  2}-{1\over  p}\right|.
\end{eqnarray}
See also     \cite{Blunck2, BDN, CCO, DN, Mi1}.

In this paper, we extend the previous result to the endpoint case $p=p_0$ and obtain the following result.

\begin{theorem}\label{main}
Suppose that $(X,d,\mu)$ is a space of homogeneous type with a dimension $n$ and that
$L$ satisfies the property $({\rm GGE}_{p_{0},p_{0}^{\prime},m})$ for some $1\leq p_{0}<2$ and $m\geq 2$.
Then for  $s_0= n\big|{1\over  2}-{1\over  p_0}\big|$,
 the operator $(I+L)^{-{s_0}}e^{itL}  $ is of weak type
 $(p_{0},p_{0})$, that is,
  there is a constant $C>0$, independent of $t$ and $f$ so that
\begin{eqnarray*}
\mu\left(\left\{x: \big|(I+L)^{-s_0}e^{itL} f(x)\big|>\alpha \right\} \right)\leq C
(1+|t|)^{n(1 - {p_0\over 2}) }
\left( {\|f\|_{p_0} \over \alpha} \right)^{p_0} ,  \ \ \ t\in{\mathbb R}
\end{eqnarray*}
 for $\alpha>0$ when $\mu(X)=\infty$ and $\alpha>\big(\|f\|_{p_{0}}/\mu(X) \big)^{p_{0}}$ when $\mu(X)<\infty$.
\end{theorem}

We would like to mention that  when $L$
satisfies the  Gaussian estimates \eqref{GE},   Chen,   Duong,  Li, Song and Yan  \cite[Theorem 1.1]{CDLSY}   proved that
 the operator $(I+L)^{-{n/2}}e^{itL}  $ is of weak type
 $(1, 1)$. In their proof, it heavily  relies  on the following Plancherel-type estimate
\begin{align}\label{Plancherel}
\int_{X}|K_{e^{-(1-i\tau)R^{-m}L}}(x,y)|^{2}d(x,y)^{s}d\mu(y)\leq C\mu(B(x,1/R))^{-1}R^{-s}(1+|\tau|)^{s},
\end{align}
for some constant $C>0$ independent of $s\geq 0$, $R>0$ and $\tau\in \mathbb{R}$, where $K_{e^{-(1-i\tau)R^{-m}L}}(x,y)$
 denotes the integral kernel of the operators $e^{-(1-i\tau)R^{-m}L}$.  In our setting, we do not have such an estimate
 at our disposal.
To overcome this difficulty,
we  apply the Phragm\'{e}n-Lindel\"{o}f theorem to show that the generalized Gaussian estimates $({\rm GGE_{p_{0},p_{0}^{\prime},m}})$ implies the following $L^{p_{0}}-L^{2}$ off-diagonal estimates of the operator $e^{-(1-i\tau)R^{-1}L}$:
\begin{align*}
\|\chi_{C_{\nu}(B)}e^{-(1-i\tau)R^{-1}L}\chi_{B}\|_{p_{0}\rightarrow 2}\leq C\mu(B(x_{B},R^{-1/m}))^{-(\frac{1}{p_{0}}-\frac{1}{2})}{\rm exp}
\left(-c\Big(\frac{\sqrt[m]{R}2^\nu r}{1+|\tau|}\Big)^{\frac{m}{m-1}}\right),
\end{align*}
for some constant $C>0$ independent of $\tau\in\mathbb{R}, R>0$ and balls $B\subset X$ with center $x_{B}$ and radius $r$. This estimate is a suitable substitute of (\ref{Plancherel}) and it helps us deduce that for any $\frac{n}{2}<s<mM$, there exists a positive constant $C$ independent of $k>k_{0}$, $t>0$, $\nu\geq \nu_{0}$, and any ball $B$ with radius $r\sim 2^{k}$, such that
\begin{align}\label{key}
\|\chi_{C_{\nu}(B)}e^{itL}F_{k}(L)\chi_{B}f\|_{2}\leq C2^{-\nu s}\mu(B)^{-(\frac{1}{p_{0}}-\frac{1}{2})}(1+|t|)^{p_{0}(\frac{1}{p_{0}}-\frac{1}{2}) (s+\frac{n}{2})}\|f\|_{p_{0}},
\end{align}
where $F_{k}(L)=(I+L)^{-(\frac{1}{p_{0}}-\frac{1}{2}) n}(I-e^{-2^{mk}L})^{M}\varphi_{0}(2^{-m(k-k_{0})/(m-1)}L)$, and $\varphi_{0}$ is a smooth function with $\supp \varphi_{0}\subset[0,1]$, $\varphi_{0}(\lambda)=1$ on $[0,1/2]$, and $k_{0}\sim 2^{-1}p_{0}\log_{2}(1+|t|)$. This inequality plays a crucial role in obtaining the sharp growth $(1+|t|)^{(1-\frac{p_{0}}{2}) n}$ in the setting of homogeneous space(see Remark \ref{rremark}).

The paper is organized as follows. In section 2 we present some $L^{p_{0}}-L^{2}$ off-diagonal estimates for heat semigroups, resolvent and compactly supported spectral multipliers. In section 3, we apply the off-diagonal estimates obtained in section 2, combined with the Hardy-Littlewood maximal function and duality arguent, to show Theorem \ref{main}.

\medskip

\section{Preliminary results}
\setcounter{equation}{0}

{We now set some notations and common concepts to be used throughout the course of the paper.}
For $1\leq p \leq+\infty$, we denote the norm of a function $f\in L^{p}(X,d\mu)$ by $\|f\|_{p}$.
 We let $\langle\cdot,\cdot\rangle$ be the scalar product in $L^{2}(X,d\mu)$.
 If $T$ is a bounded linear operator from $L^{p}(X,d\mu)$ to $L^{q}(X,d\mu)$, $1\leq p,q\leq+\infty$,
 we write $\|T\|_{p\rightarrow q}$ for the operator norm of $T$. The indicator function of a subset $E\subseteq X$ is denoted by $\chi_{E}$.
Let $f$ be a tempered distribution, then the Fourier transform $\hat{f}$ is defined by
\begin{align*}
\hat{f}(\xi):=\frac{1}{(2\pi)^{n/2}}\int_{\mathbb{R}^{n}}f(x)e^{-ix\xi}dx, \ \ \xi\in\mathbb{R}^{n}.
\end{align*}
Next, let $B(x,r)=\{y\in X,\ d(x,y)<r\}$ be the open ball with center $x\in X$ and radius $r>0$. To simplify notation we often just use $B$ instead of $B(x,r)$. Let $V_{s}$ be a multiplier operator defined by $V_{s}f(x):=\mu(B(x,s))f(x)$, and
${\mathscr M}$ will be denoted by the Hardy-Littlewood maximal function.
Also, for $1\leq p_{0}\leq 2$, let $$\sigma_{p_{0}} :=\frac{1}{p_{0}}-\frac{1}{2}.$$
For simplicity, we write
$$
C_{1}(B):=2B\ \ \  {\rm and}\ \ \ C_{\nu}(B):=2^{\nu}B-2^{\nu-1} B, \ \ \nu=2, 3, \cdots.
$$

In this section we will prove  $L^{p_{0}}-L^{2}$ off-diagonal estimates
for resolvent and compactly spectral multipliers, which
play crucial roles in obtaining the sharp growth $(1+|t|)^{p_{0}\sigma_{p_{0}} n}$ for the operator
norm $\|e^{itL}(I+L)^{-\sigma_{p_{0}} n}\|_{L^{p_{0}}\rightarrow L^{p_{0},\infty}}$ in the setting of homogeneous space.
To begin with, we show the following  $L^{p_{0}}-L^{2}$ off-diagonal estimates for heat semigroups.

\begin{lemma}\label{real}
There exists a constant $C>0$ such that for any ball $B\subset X$ with center $x_{B}$ and radius $r$ and any $\lambda>0$, $\nu\in\mathbb{N}$, the following estimate holds:
\begin{align}\label{offreal}
\|\chi_{C_{\nu}(B)}e^{-\lambda L}\chi_{B}\|_{p_{0}\rightarrow 2}\leq C\mu(B(x_{B},\lambda^{1/m}))^{-\sigma_{p_{0}}}{\rm exp}
\left(-c\Big(\frac{d(C_{\nu}(B),B)}{\lambda^{1/m}}\Big)^{m/m-1}\right).
\end{align}
\end{lemma}
\begin{proof}
The proof was essentially proved in \cite[Theorem 1.2]{Blunck1}.
We give a brief argument of this proof for completeness and the convenience of readers.

By \cite[Theorem 1.2]{Blunck1}, property $({\rm GGE}_{p_{0},p_{0}^{\prime},m})$ implies the following two ball estimate:
\begin{align*}
\|\chi_{B_{1}}V_{\lambda^{1/m}}^{a}e^{-\lambda L}V_{\lambda^{1/m}}^{b}\chi_{B_{2}}\|_{p_{0}\rightarrow 2}\leq C{\rm exp}\bigg(-c\Big(\frac{d(B_{1},B_{2})}{\lambda^{1/m}}\Big)^{\frac{m}{m-1}}\bigg), \ {\rm for\ any\ balls} \ B_{1},\ B_{2},
\end{align*}
where $a,b\geq 0$ such that $a+b=\sigma_{p_{0}}$. Therefore, estimate (\ref{offreal}) holds for $\nu=0$.
For any $\nu\geq 1$, we note that there exist $c_{n}$ balls $\{B_{\nu}^{(j)}\}_{j=1}^{c_{n}}$
such that $C_{\nu}(B)\subset \mathop{\cup}\limits_{j=1}^{c_{n}}B_{\nu}^{(j)}$ and $d(C_{\nu}(B),B)\sim d(B_{\nu}^{(j)},B)$, for all $ j=1,2\ldots,c_{n}$. Hence
\begin{align*}
\|\chi_{C_{\nu}(B)}e^{-\lambda L}\chi_{B}\|_{p_{0}\rightarrow 2}
&\leq \mu(B(x_{B},\lambda^{1/m}))^{-\sigma_{p_{0}}}
\sum_{j=1}^{c_{n}}\|\chi_{B_{\nu}^{(j)}}e^{-\lambda L}V_{\lambda^{1/m}}^{\sigma_{p_{0}}}\chi_{B}\|_{p_{0}\rightarrow 2}\\
&\leq C\mu(B(x_{B},\lambda^{1/m}))^{-\sigma_{p_{0}}}{\rm exp}\bigg(-c\Big(\frac{d(C_{\nu}(B),B)}{\lambda^{1/m}}\Big)^{\frac{m}{m-1}}\bigg).
\end{align*}

\end{proof}
\subsection{$L^{p_{0}}-L^{2}$ off-diagonal estimates for resolvent}

\begin{proposition}\label{offyujieshi}
For any $N\in \mathbb{N}$, there exists a positive constant $C$ such that for any $\nu\geq 2$,
\begin{align}\label{m1}
\|\chi_{C_{\nu}(B)}(I+L)^{-\sigma_{p_{0}} n}(I-e^{-r^{m}L})^{M}\chi_{B}\|_{p_{0}\rightarrow 2}\leq C2^{-\nu N}\mu(B)^{-\sigma_{p_{0}}},
\end{align}
and there exists a positive constant $C$ such that
\begin{align}\label{m00}
\|\chi_{2B}(I+L)^{-\sigma_{p_{0}} n}(I-e^{-r^{m}L})^{M}\chi_{2B}\|_{p_{0}\rightarrow 2}\leq C\max\{1,r^{\sigma_{p_{0}} n}\}\mu(B)^{-\sigma_{p_{0}}},
\end{align}
where $M$ is a fixed parameter chosen to be bigger than $\frac{N}{m}$ and the constant $C$ is independent of $B=B(x_{B},r)$.
\end{proposition}
\begin{proof}
Note that
$$(I+L)^{-\sigma_{p_{0}} n}=\frac{1}{\Gamma (\sigma_{p_{0}} n)}\int_{0}^{\infty}e^{-\lambda L}e^{-\lambda}\lambda^{\sigma_{p_{0}} n-1}d\lambda.
$$
From it we use the change of variables to obtain
\begin{align}\label{yujie}
(I+L)^{-\sigma_{p_{0}} n}(I-e^{-r^{m}L})^{M}
&=\frac{1}{\Gamma (\sigma_{p_{0}} n)}\int_{0}^{\infty}e^{-\lambda L}(I-e^{-r^{m}L})^{M}e^{-\lambda}\lambda^{\sigma_{p_{0}} n-1}d\lambda\nonumber\\
&=\frac{1}{\Gamma (\sigma_{p_{0}} n)}\int_{0}^{\infty}g_{r^{m}}(\lambda)e^{-\lambda L}d\lambda,
\end{align}
where
\begin{align}\label{gfunction}
g_{s}(\lambda)=\sum_{\ell=0}^{M}C_{M}^{\ell}(-1)^{\ell}\chi_{\{\lambda>\ell s\}}(\lambda)(\lambda-\ell s)^{\sigma_{p_{0}} n-1}e^{-(\lambda-\ell s)}.
\end{align}
Then, it can be verified that
\begin{eqnarray*}\label{mm} \label{1cm}
|g_{r^{m}}(\lambda)|\leq
\left
\{
\begin{array}{ll}
C\lambda^{\sigma_{p_{0}} n-1}e^{-\lambda}, &0<\lambda<r^{m},\\[4pt]
C(r^{m})^{\sigma_{p_{0}} n-1}e^{-r^{m}}+(\lambda-\beta r^{m})^{\sigma_{p_{0}} n-1}e^{-(\lambda-\beta r^{m})},
&\beta r^{m}\leq \lambda< (\beta+1)r^{m}, 1\leq \beta\leq M, \\[4pt]
Cr^{mM}\lambda^{\sigma_{p_{0}} n-1-M}e^{-\frac{\lambda}{2(M+1)}}, &\lambda\geq (M+1)r^{m},\end{array}\right.
\end{eqnarray*}




Now let us prove (\ref{m1}). By the formula (\ref{yujie}), Lemma \ref{real} and doubling condition (\ref{doubling}),
 we get
\begin{eqnarray*}
&&\hspace{-1.2cm}\left\|\chi_{C_{\nu}(B)}(I+L)^{-\sigma_{p_{0}} n}(I-e^{-r^{m}L})^{M}\chi_{B}\right\|_{p_{0}\rightarrow 2}\\
&\leq& C\mu(B)^{-\sigma_{p_{0}} }
 \int_{0}^{\infty}|g_{r^{m}}(\lambda)|\left(1+\frac{r}{\lambda^{1/m}}\right)^{\sigma_{p_{0}} n}
 {\rm exp}\left(-c\Big(\frac{2^\nu r}{\lambda^{1/m}}\Big)^{\frac{m}{m-1}}\right)d\lambda\\
&=& C\mu(B)^{-\sigma_{p_{0}} }\left(\int_{0}^{r^{m}}+ \int_{r^{m}}^{(M+1)r^{m}}+\int_{(M+1)r^{m}}^{\infty}\right) |g_{r^{m}}(\lambda)|\left(1+\frac{r}{\lambda^{1/m}}\right)^{\sigma_{p_{0}} n}
 {\rm exp}\left(-c\Big(\frac{2^\nu r}{\lambda^{1/m}}\Big)^{\frac{m}{m-1}}\right) d\lambda
\\
& =:&I+II+III.
\end{eqnarray*}
For the term $I,$ we use the property of $ g_m $ to obtain
\begin{align}\label{I}
I
&\leq C\mu(B)^{-\sigma_{p_{0}} }\int_{0}^{r^{m}}\lambda^{\sigma_{p_{0}} n-1}e^{-\lambda}\Big(1+\frac{r}{\lambda^{1/m}}\Big)^{\sigma_{p_{0}} n}{\rm exp}\bigg(-c\Big(\frac{2^\nu r}{\lambda^{1/m}}\Big)^{\frac{m}{m-1}}\bigg)d\lambda\nonumber\\
&\leq C2^{-\nu N}\mu(B)^{-\sigma_{p_{0}}}\int_{0}^{1}\lambda^{-1+\frac{N-\sigma_{p_{0}} n}{m}}d\lambda \nonumber\\
 &\leq C2^{-\nu N}\mu(B)^{-\sigma_{p_{0}}},
\end{align}
for some large  $N>\sigma_{p_{0}} n$.  For the term $II$, we have
\begin{align}\label{II}
II
&\leq C\mu(B)^{-\sigma_{p_{0}} }\sum_{\ell=1}^{M}\int_{\ell r^{m}}^{(\ell+1)r^{m}}(r^{m})^{\sigma_{p_{0}} n-1}e^{-r^{m}}\Big(1+\frac{r}{\lambda^{1/m}}\Big)^{\sigma_{p_{0}} n}{\rm exp}\bigg(-c\Big(\frac{2^\nu r}{\lambda^{1/m}}\Big)^{\frac{m}{m-1}}\bigg)d\lambda\nonumber\\
&+C\mu(B)^{-\sigma_{p_{0}} }\sum_{\ell=1}^{M}\int_{\ell r^{m}}^{(\ell+1)r^{m}}(\lambda-\ell r^{m})^{\sigma_{p_{0}} n-1}e^{-(\lambda-\ell r^{m})}\Big(1+\frac{r}{\lambda^{1/m}}\Big)^{\sigma_{p_{0}} n}{\rm exp}\bigg(-c\Big(\frac{2^\nu r}{\lambda^{1/m}}\Big)^{\frac{m}{m-1}}\bigg)d\lambda\nonumber\\
&\leq C2^{-\nu N}\mu(B)^{-\sigma_{p_{0}}}.
\end{align}
Consider  the term $III$. Since $M>\frac{N}{m}$, we conclude that
\begin{align*}
III
&\leq C\mu(B)^{-\sigma_{p_{0}}}\int_{(M+1)r^{m}}^{\infty}r^{mM}\lambda^{\sigma_{p_{0}} n-1-M}e^{-\frac{\lambda}{2(M+1)}}\Big(1+\frac{r}{\lambda^{1/m}}\Big)^{\sigma_{p_{0}} n}{\rm exp}\bigg(-c\Big(\frac{2^\nu r}{\lambda^{1/m}}\Big)^{\frac{m}{m-1}}\bigg)d\lambda\\
&\leq C2^{-\nu N}\mu(B)^{-\sigma_{p_{0}}}\int_{1}^{\infty}\lambda^{-1-M+\frac{N}{m}}d\lambda\\
&\leq C2^{-\nu N}\mu(B)^{-\sigma_{p_{0}}}.
\end{align*}
This, in combination with  estimates (\ref{I}) and (\ref{II}),  shows  the desired estimate (\ref{m1}).

To show estimate (\ref{m00}), we see that
\begin{eqnarray*}
&&\hspace{-1.2cm}\left\|\chi_{2B}(I+L)^{-\sigma_{p_{0}} n}(I-e^{-r^{m}L})^{M}\chi_{2B}\right\|_{p_{0}\rightarrow 2}\\
&\leq& C\mu(B)^{-\sigma_{p_{0}} }
 \int_{0}^{\infty}|g_{r^{m}}(\lambda)|\left(1+\frac{r}{\lambda^{1/m}}\right)^{\sigma_{p_{0}} n}d\lambda\\
&=& C\mu(B)^{-\sigma_{p_{0}} }\left(\int_{0}^{r^{m}}+ \int_{r^{m}}^{(M+1)r^{m}}+\int_{(M+1)r^{m}}^{\infty}\right) |g_{r^{m}}(\lambda)|\left(1+\frac{r}{\lambda^{1/m}}\right)^{\sigma_{p_{0}} n} d\lambda.
\end{eqnarray*}
Then, we can deduce that the last two parts are no larger than $C\mu(B)^{-\sigma_{p_{0}}}$ by the same way.
 Also, It can be shown by a simple modification of the estimate of $III$ that the first part can be
 bounded by $C\max\{1,r^{\sigma_{p_{0}} n}\}\mu(B)^{-\sigma_{p_{0}}}$. We omit the details  and leave it to the readers.
%

\end{proof}
\subsection{$L^{p_{0}}-L^{2}$ off-diagonal estimates for compactly supported spectral multipliers}
To begin with, we state the following version of Phragmen-Lindel\"{o}f Theorem.
\begin{lemma}\label{PL0}
Suppose that function $F$ is analytic in $\mathbb{C}_{+}:=\{z\in\mathbb{C}:{\rm Rez}>0\}$ and that
\begin{align*}
&|F(z)|\leq a_{1}({\rm Re}z)^{-\beta_{1}},\\
&|F(\lambda)|\leq a_{1}t^{-\beta_{1}}{\rm exp}(-a_{2}t^{-\beta_{2}}),
\end{align*}
for some $a_{1}$,$a_{2}>0$, $\beta_{1}\geq 0$, $\beta_{2}\in(0,1]$, all $t>0$ and $z\in\mathbb{C}_{+}$. Then
\begin{align*}
|F(z)|\leq a_{1}2^{\beta_{1}}({\rm Re}z)^{-\beta_{1}}{\rm exp}\bigg(-\frac{a_{2}\beta_{2}}{2}|z|^{-\beta_{2}-1}{\rm Re}z\bigg)
\end{align*}
for all $z\in\mathbb{C}_{+}$.
\end{lemma}
\begin{proof}
For the proof, we refer it to \cite[Lemma 9]{Davies}.
\end{proof}

By Lemma \ref{PL0}, we can modify the argument from \cite[Lemma 3.3]{hardy} to extend the off-diagonal estimate (\ref{offreal}) from real times $t>0$ to complex times $z=-(1-i\tau)R^{-1}$ for any $R>0$.
\begin{lemma}\label{tt1}
There exists a constant $C>0$ such that for any ball $B\subset X$ with center $x_{B}$ and radius $r$, integer $\nu\geq 2$, real number $R>0$, the following estimate holds:
\begin{align}\label{t1}
\|\chi_{C_{\nu}(B)}e^{-(1-i\tau)R^{-1}L}\chi_{B}\|_{p_{0}\rightarrow 2}\leq C\mu(B(x_{B},R^{-1/m}))^{-\sigma_{p_{0}}}{\rm exp}\bigg(-c\Big(\frac{\sqrt[m]{R}2^\nu r}{1+|\tau|}\Big)^{\frac{m}{m-1}}\bigg).
\end{align}
\end{lemma}
\begin{proof}
For simplicity, we write $z=-(1-i\tau)R^{-1}$.
For any ball $B\subset X$, consider the analytic function $F:\mathbb{C}_{+}\rightarrow \mathbb{R}$ defined by
\begin{align*}
F(z):=e^{-Rz}\mu(B(x_{B},R^{-1/m}))^{\sigma_{p_{0}}}\langle e^{-zL}f_{1},f_{2}\rangle,
\end{align*}
where supp$f_{1}\subset B$ and supp$f_{2}\subset C_{\nu}(B)$.

It was shown in \cite[Proposition 3.1]{Blunck2} that
\begin{align*}
\|e^{-\frac{{\rm Re}z}{4}L}V_{({\rm Re}z)^{1/m}}^{\sigma_{p_{0}}}\|_{p_{0}\rightarrow 2}\leq C.
\end{align*}

This, together with the spectral theorem, shows that
\begin{align*}
\|e^{-zL}V_{({\rm Re}z)^{1/m}}^{\sigma_{p_{0}}}\|_{p_{0}\rightarrow 2}\leq\|e^{-\frac{{\rm Re}z}{4}L}\|_{2\rightarrow 2}\|e^{-(z-\frac{{\rm Re}z}{2})L}\|_{2\rightarrow 2}\|e^{-\frac{{\rm Re}z}{4}L}V_{({\rm Re}z)^{1/m}}^{\sigma_{p_{0}}}\|_{p_{0}\rightarrow 2}\leq C.
\end{align*}
Hence,
\begin{align}
|\langle e^{-zL}f_{1},f_{2}\rangle|
\leq \|e^{-zL}V_{({\rm Re}z)^{1/m}}^{\sigma_{p_{0}}}\|_{p_{0}\rightarrow 2}\|V_{({\rm Re}z)^{1/m}}^{-\sigma_{p_{0}}}f_{1}\|_{p_{0}}\|f_{2}\|_{2}
\leq C\mu(B(x_{B},({\rm Re}z)^{1/m}))^{-\sigma_{p_{0}}}\|f_{1}\|_{p_{0}}\|f_{2}\|_{2}.
\end{align}
This, in combination with doubling condition (\ref{doubling}), yields
\begin{align}\label{v3}
|F(z)|
&\leq Ce^{-R{\rm Re}z}\bigg(\frac{\mu(B(x_{B},R^{-1/m}))}{\mu(B(x_{B},({\rm Re}z)^{1/m}))}\bigg)^{\sigma_{p_{0}}}\|f_{1}\|_{p_{0}}\|f_{2}\|_{2}\nonumber\\
&\leq Ce^{-R{\rm Re}z}\Big(1+\frac{1}{R{\rm Re}z}\Big)^{\frac{\sigma_{p_{0}} n}{m}}\|f_{1}\|_{p_{0}}\|f_{2}\|_{2}\nonumber\\
&\leq C(R{\rm Re}z)^{\frac{-\sigma_{p_{0}} n}{m}}\|f_{1}\|_{p_{0}}\|f_{2}\|_{2}.
\end{align}

Similarly, by Lemma \ref{real},
\begin{align}\label{v4}
|F(\lambda)|
&\leq e^{-R\lambda}\mu(B(x_{B},R^{-1/m}))^{\sigma_{p_{0}}}\|\chi_{C_{\nu}(B)}e^{-\lambda L}\chi_{B}\|_{p_{0}\rightarrow 2}\|f_{1}\|_{p_{0}}\|f_{2}\|_{2}\nonumber\\
&\leq Ce^{-R\lambda}\bigg(\frac{\mu(B(x_{B},R^{-1/m}))}{\mu(B(x_{B},\lambda^{1/m}))}\bigg)^{\sigma_{p_{0}}}{\rm exp}\bigg(-c\Big(\frac{2^\nu r}{\lambda^{1/m}}\Big)^{\frac{m}{m-1}}\bigg)\|f_{1}\|_{p_{0}}\|f_{2}\|_{2}\nonumber\\
&\leq Ce^{-R\lambda}\Big(1+\frac{1}{R\lambda}\Big)^{\frac{\sigma_{p_{0}} n}{m}}{\rm exp}\bigg(-c\Big(\frac{2^\nu r}{\lambda^{1/m}}\Big)^{\frac{m}{m-1}}\bigg)\|f_{1}\|_{p_{0}}\|f_{2}\|_{2}\nonumber\\
&\leq C(R\lambda)^{-\frac{\sigma_{p_{0}} n}{m}}{\rm exp}\bigg(-c\Big(\frac{2^\nu r}{\lambda^{1/m}}\Big)^{\frac{m}{m-1}}\bigg)\|f_{1}\|_{p_{0}}\|f_{2}\|_{2}.
\end{align}

Next, combining (\ref{v3}) with (\ref{v4}), we choose $z=(1-i\tau)R^{-1}$, $a_{1}=CR^{-\frac{\sigma_{p_{0}} n}{m}}\|f_{1}\|_{p_{0}}\|f_{2}\|_{2}$, $a_{2}=c(2^\nu r)^{\frac{m}{m-1}}$, $\beta_{1}=\frac{\sigma_{p_{0}} n}{m}$ and $\beta_{2}=1/(m-1)$ in Lemma \ref{PL0} to obtain
\begin{align*}
|F((1-i\tau)R^{-1})|\leq C{\rm exp}\bigg(-c\Big(\frac{\sqrt[m]{R}2^\nu r}{1+|\tau|}\Big)^{\frac{m}{m-1}}\bigg)\|f_{1}\|_{p_{0}}\|f_{2}\|_{2},
\end{align*}
which yields the estimate (\ref{t1}).
\end{proof}

We define a Besov type norm of $F$ by
\begin{align*}
\|F\|_{B^{s}}:=\int_{-\infty}^{\infty}|\hat{F}(\tau)|(1+|\tau|)^{s}d\tau,
\end{align*}
where $\hat{F}$ denotes the Fourier transform of $F$. Applying Fubini theorem, we can easily check that for every functions $F$ and $G$,
\begin{align}\label{algebra}
\|FG\|_{B^{s}}
&=\int_{-\infty}^{\infty}|\widehat{FG}(\tau)|(1+|\tau|)^{s}d\tau\nonumber\\
&\leq \int_{-\infty}^{\infty}\int_{-\infty}^{\infty}|\hat{F}(\tau-\eta)\hat{G}(\eta)|(1+|\tau-\eta|)^{s}(1+|\eta|)^{s}d\eta d\tau\nonumber\\
&\leq \|F\|_{B^{s}}\|G\|_{B^{s}}.
\end{align}

In the sequel, for any $R>0$, we denote the dilation of a function $F$ by $\delta_{R}F(\cdot):=F(R\cdot)$.
\begin{proposition}\label{besov4}
For every $s\geq 0$, there exists a constant $C>0$ such that for every $\nu\geq 2$,
\begin{align}\label{besov2}
\|\chi_{C_{\nu}(B)}F(L)\chi_{B}\|_{p_{0}\rightarrow 2}\leq C\mu(B(x_{B},R^{-1/m}))^{-\sigma_{p_{0}}}(\sqrt[m]{R}2^\nu r)^{-s}\|\delta_{R}F\|_{B^{s}}
\end{align}
for all balls $B:=B(x_{B},r)\subseteq X$, and all Borel functions $F$ such that supp$F\subseteq [-R,R]$.
\end{proposition}
\begin{proof}
Let $G(\lambda)=(\delta_{R}F)(\lambda)e^{\lambda}$. In virtue of the Fourier inversion formula
\begin{align*}
F(L)=G(L/R)e^{-L/R}=\frac{1}{2\pi}\int_{\mathbb{R}}e^{(i\tau-1)R^{-1}L}\hat{G}(\tau)d\tau.
\end{align*}
This, in combination with estimate (\ref{t1}) yields that
\begin{align*}
\|\chi_{C_{\nu}(B)}F(L)\chi_{B}\|_{p_{0}\rightarrow 2}
&\leq\frac{1}{2\pi}\int_{\mathbb{R}}\|\chi_{C_{\nu}(B)}e^{(i\tau-1)R^{-1}L}\chi_{B}\|_{p_{0}\rightarrow 2}|\hat{G}(\tau)|d\tau\\
&\leq C\mu(B(x_{B},R^{-1/m}))^{-\sigma_{p_{0}}}\int_{\mathbb{R}}{\rm exp}\bigg(-c\Big(\frac{\sqrt[m]{R}2^\nu r}{1+|\tau|}\Big)^{\frac{m}{m-1}}\bigg)|\hat{G}(\tau)|d\tau\\
&\leq C\mu(B(x_{B},R^{-1/m}))^{-\sigma_{p_{0}}}(\sqrt[m]{R}2^\nu r)^{-s}\|G\|_{B^{s}}.
\end{align*}

Note that supp$\delta_{R}F\subseteq[-1,1]$. Thus by taking a smooth cutoff function $\psi$ such that supp$\psi\subset[-2,2]$ and $\psi(\lambda)=1$ for $\lambda\in[-1,1]$, we have
\begin{align*}
G(\lambda)=(\delta_{R}F)(\lambda)e^{\lambda}=(\delta_{R}F)(\lambda)\psi(\lambda)e^{\lambda}.
\end{align*}
Hence, by (\ref{algebra}),
\begin{align*}
\|G\|_{B^{s}}\leq C\|\delta_{R}F\|_{B^{s}}\|\psi(\lambda)e^{\lambda}\|_{B^{s}}\leq C\|\delta_{R}F\|_{B^{s}}.
\end{align*}
This ends the proof of proposition \ref{besov4}.
\end{proof}

\begin{remark}\label{rremark}
{\rm In  \cite[Lemma 2.5]{KU}, the author used some techniques introduced by Blunck\cite{Blunck2} to show that for any $\nu\geq 2$, $e^{-zL}$ satisfies the following off-diagonal estimate:
\begin{align*}
\|\chi_{C_{\nu}(B)}e^{-zL}\chi_{B}\|_{p_{0}\rightarrow 2}\leq C\frac{1}{\mu(B(x_{B},({\rm Re}z))^{\frac{1}{m}-1}|z|)^{\sigma_{p_{0}}}}\bigg(\frac{|z|}{{\rm Re}z}\bigg)^{\sigma_{p_{0}} n}2^{\nu n}{\rm exp}\bigg(-c\Big(\frac{2^\nu r}{({\rm Re}z)^{\frac{1}{m}-1}|z|}\Big)^{\frac{m}{m-1}}\bigg).
\end{align*}
It follows that for any $\nu\geq 2$, $R>0$, $\tau\in\mathbb{R}$,
\begin{align}\label{nn1}
\|\chi_{C_{\nu}(B)}e^{-(1-i\tau)R^{-1}L}\chi_{B}\|_{p_{0}\rightarrow 2}\leq C\frac{1}{\mu(B(x_{B},R^{-1/m}\sqrt{1+\tau^{2}}))^{\sigma_{p_{0}}}}(1+\tau^{2})^{\frac{\sigma_{p_{0}} n}{2}}2^{\nu n}{\rm exp}\bigg(-c\Big(\frac{2^\nu r}{\sqrt{1+\tau^{2}}}\Big)^{\frac{m}{m-1}}\bigg),
\end{align}
for all $B=B(x_{B},r)\subseteq X$. In our Lemma \ref{tt1}, we made an important improvement in
 obtaining the upper bound on the right hand side of (\ref{nn1}) without the factor ``$2^{\nu n}$", which plays a key role in obtaining the sharp growth $(1+|t|)^{p_{0}\sigma_{p_{0}} n}$ in the proof of Theorem \ref{main}. Also, in the setting of homogeneous space, it seems hard to remove $\sqrt{1+\tau^{2}}$ directly from the volume term $\mu(B(x_{B},R^{-1/m}\sqrt{1+\tau^{2}}))^{\sigma_{p_{0}}}$ appeared on the right hand side of (\ref{nn1}) by the doubling condition (\ref{doubling}). Instead, our lemma \ref{tt1} provide a slight different but much more subtle upper bound such that the factor $\sqrt{1+\tau^{2}}$ doesn't appear on the volume term, which is helpful to obtain off-diagonal estimates (\ref{besov2}) of compactly supported spectral multipliers.}
\end{remark}

\medskip

\section{Proof of theorem \ref{main}}
\setcounter{equation}{0}

Fix $f\in L^{p_0}(X)$. For $\alpha>\mu(X)^{-p_{0}}\|f\|_{p_{0}}^{p_{0}}$, we apply the Calder\'{o}n-Zygmund
decomposition at height $\alpha$ to $|f|$. Then there exist constants $C$ and $K$ so that

\smallskip
${\rm (i)}f=g+b=g+\sum_{j}b_{j};$

\smallskip
${\rm (ii)}\|g\|_{p_{0}}\leq C\|f\|_{p_{0}}, \|g\|_{\infty}\leq C\alpha;$

\smallskip
${\rm (iii)}b_{j}$ is supported in $B_{j}$ and $\#\{j:x\in 2B_{j}\}\leq K$ for all $x\in X;$

\smallskip
${\rm (iv)}\int_{X}|b_{j}|^{p_{0}}d\mu\leq C\alpha^{p_{0}}\mu(B_{j})$ and $\sum_{j}\mu(B_{j})\leq C\alpha^{-p_{0}}\|f\|_{p_{0}}^{p_{0}}$.

Let $r_{B_{j}}$ be the radius of $B_{j}$ and let
\begin{align*}
J_{k}=\{j:2^{k}\leq r_{B_{j}}<2^{k+1}\},\  {\rm for} \ k\in \mathbb{Z}.
\end{align*}
We decompose the bad function $b(x)$ as follows
\begin{align*}
b(x)=\sum_{k\leq k_{0}}\sum_{j\in J_{k}}b_{j}(x)+\sum_{k> k_{0}}\sum_{j\in J_{k}}b_{j}(x)=:h_{1}(x)+h_{2}(x),
\end{align*}
where $k_{0}$ is an integer such that $2^{k_{0}}\leq (1+|t|)^{\frac{p_{0}}{2}}< 2^{k_{0}+1}$. Then it is enough to show that there exists a constant $C>0$ independent of $\alpha$ and $t$ such that
\begin{align}\label{good}
\mu\Big(\Big\{x:\big|e^{itL}(I+L)^{-\sigma_{p_{0}} n}g(x)\big|>\alpha\Big\}\Big)\leq C\alpha^{-p_{0}}(1+|t|)^{p_{0}\sigma_{p_{0}} n}\|f\|_{{p_{0}}}^{p_{0}}
\end{align}
and such that for $i=1,2$,
\begin{align}\label{bad}
\mu\Big(\Big\{x:\big|e^{itL}(I+L)^{-\sigma_{p_{0}} n}h_{i}(x)\big|>\alpha\Big\}\Big)\leq C\alpha^{-p_{0}}(1+|t|)^{p_{0}\sigma_{p_{0}} n}\|f\|_{{p_{0}}}^{p_{0}}.
\end{align}

By the property (ii) and spectral theory,
$$\mu\Big(\Big\{x:\big|e^{itL}(I+L)^{-\sigma_{p_{0}} n}g(x)\big|>\alpha\Big\}\Big)\leq \alpha^{-2}\|e^{itL}(I+L)^{-\sigma_{p_{0}} n}g\|_{2}^{2}\leq \alpha^{-2}\|g\|_{2}^{2}\leq C\alpha^{-p_{0}}\|f\|_{p_{0}}^{p_{0}},$$
which proves (\ref{good}).

\medskip

\noindent
\textbf{Proof of (\ref{bad}) for $i=1$.}

\smallskip

Since the Schr\"odinger group $e^{-itL}$ is bounded on $L^{2}(X)$, we have
\begin{align*}
\mu\Big(\Big\{x:\big|e^{itL}(I+L)^{-\sigma_{p_{0}} n}h_{1}(x)\big|>\alpha\Big\}\Big)
&\leq \alpha^{-2}\|e^{itL}(I+L)^{-\sigma_{p_{0}} n}h_{1}\|_{2}^{2}\\
&\leq \alpha^{-2}\|(I+L)^{-\sigma_{p_{0}} n}h_{1}\|_{2}^{2}\\
&\leq \alpha^{-2}\Big{\|}\sum_{k\leq k_{0}}\sum_{j\in J_{k}}(I+L)^{-\sigma_{p_{0}} n}(I-e^{-r_{B_{j}}^{m}L})^{M}b_{j}(x)\Big{\|}_{2}^{2}\\
&+\alpha^{-2}\Big{\|}\sum_{k\leq k_{0}}\sum_{j\in J_{k}}(I+L)^{-\sigma_{p_{0}} n}[I-(I-e^{-r_{B_{j}}^{m}L})^{M}]b_{j}(x)\Big{\|}_{2}^{2}\\
&=:I+II.
\end{align*}
Now we borrow the argument from \cite{Aus} to estimate these two parts. We first estimate the term $I$ by duality:
\begin{align}\label{duiou1}
I&=\alpha^{-2}\sup\limits_{\|u\|_{2}=1}\bigg{(}\int_{X}u(x)\sum_{k\leq k_{0}}\sum_{j\in J_{k}}(I+L)^{-\sigma_{p_{0}} n}(I-e^{-r_{B_{j}}^{m}L})^{M}b_{j}(x)d\mu(x)\bigg{)}^{2}\nonumber\\
&\leq \alpha^{-2}\sup\limits_{\|u\|_{2}=1}\bigg{(}\sum_{k\leq k_{0}}\sum_{j\in J_{k}}\sum_{\nu=1}^{\infty}A_{\nu j}\bigg{)}^{2},
\end{align}
where $A_{\nu j}:=\int_{C_{\nu}(B_{j})}|(I+L)^{-\sigma_{p_{0}} n}(I-e^{-r_{B_{j}}^{m}L})^{M}b_{j}(x)||u(x)|d\mu(x)$.
By Proposition \ref{offyujieshi} and the fact that $r_{B_{j}}\leq 2^{k_{0}+1}\leq 2(1+|t|)^{\frac{p_{0}}{2}}$, for any $\nu\geq 1$,
\begin{align}\label{claim}
\|(I+L)^{-\sigma_{p_{0}} n}(I-e^{-r_{B_{j}}^{m}L})^{M}b_{j}\|_{L^{2}(C_{\nu}(B_{j}))}
&\leq \|\chi_{C_{\nu}(B_{j})}(I+L)^{-\sigma_{p_{0}} n}(I-e^{-r_{B_{j}}^{m}L})^{M}\chi_{B_{j}}\|_{p_{0}\rightarrow 2}\|b_{j}\|_{p_{0}}\nonumber\\
&\leq C(1+|t|)^{\frac{p_{0}\sigma_{p_{0}} n}{2}}2^{-\nu N}\alpha \mu(B_{j})^{\frac{1}{2}}.
\end{align}
Also  for any $y\in B_{j}$ and any $\nu \geq 1$,
\begin{align}\label{max}
\|u\|_{L^{2}(C_{\nu}(B_{j}))}
\leq \|u\|_{L^{2}((\nu+1)B_{j})}
\leq \mu\big((\nu+1)B_{j}\big)^{\frac{1}{2}}\mathscr {M}(|u|^{2})(y)^{\frac{1}{2}}
\leq C2^{\frac{\nu n}{2}}\mu(B_{j})^{\frac{1}{2}}\mathscr {M}(|u|^{2})(y)^{\frac{1}{2}},
\end{align}
where in the last inequality we used the doubling condition (\ref{doubling}).
We apply H\"{o}lder's inequality, one obtains
\begin{align*}
A_{\nu j}
\leq \|(I+L)^{-\sigma_{p_{0}} n}(I-e^{-r_{B_{j}}^{m}L})^{M}b_{j}\|_{L^{2}(C_{\nu}(B_{j}))}\|u\|_{L^{2}(C_{\nu}(B_{j}))}
\leq C(1+|t|)^{\frac{p_{0}\sigma_{p_{0}} n}{2}}2^{-(N-\frac{n}{2})\nu}\alpha \mu(B_{j})\mathscr {M}(|u|^{2})(y)^{\frac{1}{2}}.
\end{align*}
Averaging over $B_{j}$ yields
\begin{align*}
A_{\nu j}
&\leq C(1+|t|)^{\frac{p_{0}\sigma_{p_{0}} n}{2}}2^{-(N-\frac{n}{2})\nu}\alpha \int_{B_{j}}\mathscr {M}(|u|^{2})(y)^{\frac{1}{2}}d\mu(y).
\end{align*}
Choosing $N$ sufficient large and then Summing over $\nu\geq 1$ and $j$, we have
\begin{align*}
I
&\leq C(1+|t|)^{p_{0}\sigma_{p_{0}} n}\sup\limits_{\|u\|_{2}=1}\bigg(\int_{\mathop{\cup}\limits_{j}B_{j}}\mathscr {M}(|u|^{2})(y)^{\frac{1}{2}}d\mu(y)\bigg)^{2}\\
&\leq C(1+|t|)^{p_{0}\sigma_{p_{0}} n}\sup\limits_{\|u\|_{2}=1}\mu\Big(\mathop{\cup}\limits_{j}B_{j}\Big)\|u^{2}\|_{1}\\
&\leq C(1+|t|)^{p_{0}\sigma_{p_{0}} n}\alpha^{-p_{0}}\|f\|_{p_{0}}^{p_{0}},
\end{align*}
where in the next to last inequality, we use Kolmogorov's lemma and the weak type (1,1) of the Hardy-Littlewood maximal function, and in the last one we apply the property (iv) in the Calder\'{o}n-Zygmund decomposition.

Next, we estimate the second part $II$, by spectral theorem,
\begin{align*}
II&=\alpha^{-2}\Big{\|}\sum_{k\leq k_{0}}\sum_{j\in J_{k}}(I+L)^{-\sigma_{p_{0}} n}[I-(I-e^{-r_{B_{j}}^{m}L})^{M}]b_{j}(x)\Big{\|}_{2}^{2}\\
&\leq\alpha^{-2}\Big{\|}\sum_{k\leq k_{0}}\sum_{j\in J_{k}}[I-(I-e^{-r_{B_{j}}^{m}L})^{M}]b_{j}(x)\Big{\|}_{2}^{2}
\end{align*}
Observe that $[I-(I-e^{-r_{B_{j}}^{m}L})^{M}]$ is a finite combination of the terms $e^{-jr_{B_{j}}^{m}L}$, $j=1,\ldots, M$ and that, by the doubling condition (\ref{doubling}) and Lemma \ref{real}, semigroup $e^{-jr_{B_{j}}^{m}L}$ satisfies the following estimate
\begin{align*}
\|\chi_{C_{\nu}(B_{j})}e^{-jr_{B_{j}}^{m}L}\chi_{B_{j}}\|_{p_{0}\rightarrow 2}\leq C2^{-\nu N}\mu(B_{j})^{-\sigma_{p_{0}}}.
\end{align*}
We can easily apply the same duality argument as estimating the term $I$ to show that
\begin{align*}
II\leq C\alpha^{-p_{0}}\|f\|_{p_{0}}^{p_{0}}.
\end{align*}

Combining the estimates for $I$ and $II$, we obtain (\ref{bad}) for $i=1$, i.e.,
\begin{align*}
\mu\Big(\Big\{x:\big|e^{itL}(I+L)^{-\sigma_{p_{0}} n}h_{1}(x)\big|>\alpha\Big\}\Big)\leq C\alpha^{-p_{0}}(1+|t|)^{p_{0}\sigma_{p_{0}} n}\|f\|_{{p_{0}}}^{p_{0}}.
\end{align*}

\medskip

\noindent
\textbf{Proof of (\ref{bad}) for $i=2$.}

\smallskip

Set $\Omega_{t}:=\mathop{\cup}\limits_{j}2(1+|t|)^{p_{0}\sigma_{p_{0}}}B_{j}$. By (iv) in the Calder\'{o}n-Zygmund decomposition,
\begin{align*}
\mu\Big(\Big\{x\in\Omega_{t}:\big|e^{itL}(I+L)^{-\sigma_{p_{0}} n}h_{2}(x)\big|>\alpha\Big\}\Big)
&\leq C\mu(\mathop{\cup}\limits_{j}2(1+|t|)^{p_{0}\sigma_{p_{0}}}B_{j})\\
&\leq C(1+|t|)^{p_{0}\sigma_{p_{0}} n}\sum_{j}\mu(B_{j})\\
&\leq C\alpha^{-p_{0}}(1+|t|)^{p_{0}\sigma_{p_{0}} n}\|f\|_{p_{0}}^{p_{0}}.
\end{align*}

Next we show that
\begin{align*}
\mu\Big(\Big\{x\in\Omega_{t}^{c}:\big|e^{itL}(I+L)^{-\sigma_{p_{0}} n}h_{2}(x)\big|>\alpha\Big\}\Big)\leq C\alpha^{-p_{0}}(1+|t|)^{p_{0}\sigma_{p_{0}} n}\|f\|_{{p_{0}}}^{p_{0}}.
\end{align*}
Note that for every $j\in J_{k}$, $k>k_{0}$, the function $b_{j}$ is supported in $B_{j}$, and the radius of the ball $B_{j}$ is equivalent to $2^{k}$. We decompose $(I+L)^{-\sigma_{p_{0}} n}e^{itL}b_{j}$ into two parts:
\begin{align}\label{b1}
e^{itL}(I+L)^{-\sigma_{p_{0}} n}b_{j}
=e^{itL}(I+L)^{-\sigma_{p_{0}} n}[I-(I-e^{-2^{mk}L})^{M}]b_{j}
+e^{itL}(I+L)^{-\sigma_{p_{0}} n}(I-e^{-2^{mk}L})^{M}b_{j},
\end{align}
where $M$ is a fixed parameter chosen to be bigger than $\frac{n}{2m}$.

Consider the term $e^{itL}(I+L)^{-\sigma_{p_{0}} n}[I-(I-e^{-2^{mk}L})^{M}]b_{j}$.
By the spectral theorem,
\begin{align*}
&\mu\Big(\Big\{x\in\Omega_{t}^{c}:\big|\sum_{k> k_{0}}\sum_{j\in J_{k}}e^{itL}(I+L)^{-\sigma_{p_{0}} n}[I-(I-e^{-2^{mk}L})^{M}]b_{j}(x)\big|>\alpha\Big\}\Big)\\
\leq &\alpha^{-2}\Big\|\sum_{k> k_{0}}\sum_{j\in J_{k}}e^{itL}(I+L)^{-\sigma_{p_{0}} n}[I-(I-e^{-2^{mk}L})^{M}]b_{j}(x)\Big\|_{2}^{2}\\
\leq &\alpha^{-2}\Big\|\sum_{k> k_{0}}\sum_{j\in J_{k}}[I-(I-e^{-2^{mk}L})^{M}]b_{j}(x)\Big\|_{2}^{2}.
\end{align*}
Then we  follow the similar procedure as estimating the term $II$ to show that
\begin{align*}
\mu\Big(\Big\{x\in\Omega_{t}^{c}:\big|\sum_{k> k_{0}}\sum_{j\in J_{k}}e^{itL}(I+L)^{-\sigma_{p_{0}} n}[I-(I-e^{-2^{mk}L})^{M}]b_{j}(x)\big|>\alpha\Big\}\Big)\leq C\alpha^{-p_{0}}\|f\|_{p_{0}}^{p_{0}}.
\end{align*}
For  the term $e^{itL}(I+L)^{-\sigma_{p_{0}} n}(I-e^{-2^{mk}L})^{M}b_{j}$ in (\ref{b1}), we let $\varphi_{1}$ be a smooth function such that ${\rm supp}\varphi_{1}\subset[1/2,\infty]$ and $\varphi_{1}(\lambda)=1$ on $[1,\infty]$. Also, let $\varphi_{0}(\lambda)=1-\varphi_{1}(\lambda)$. We further decompose this part as follow
\begin{align*}
e^{itL}(I+L)^{-\sigma_{p_{0}} n}(I-e^{-2^{mk}L})^{M}b_{j}=e^{itL}F_{k}(L)b_{j}+e^{itL}G_{k}(L)b_{j},
\end{align*}
where
\begin{align*}
F_{k}(L)=(I+L)^{-\sigma_{p_{0}} n}(I-e^{-2^{mk}L})^{M}\varphi_{0}(2^{-m(k-k_{0})/(m-1)}L)
\end{align*}
and
\begin{align*}
G_{k}(L)=(I+L)^{-\sigma_{p_{0}} n}(I-e^{-2^{mk}L})^{M}\varphi_{1}(2^{-m(k-k_{0})/(m-1)}L).
\end{align*}
Hence
\begin{align*}
\mu\Big(\Big\{x\in\Omega_{t}^{c}:&\big|\sum_{k> k_{0}}\sum_{j\in J_{k}}e^{itL}(I+L)^{-\sigma_{p_{0}} n}(I-e^{-2^{mk}L})^{M}b_{j}(x)\big|>\alpha\Big\}\Big)\\
&\leq \mu\Big(\Big\{x\in\Omega_{t}^{c}:\big|\sum_{k> k_{0}}\sum_{j\in J_{k}}e^{itL}F_{k}(L)b_{j}(x)\big|>\frac{\alpha}{2}\Big\}\Big)\\
&+ \mu\Big(\Big\{x\in\Omega_{t}^{c}:\big|\sum_{k> k_{0}}\sum_{j\in J_{k}}e^{itL}G_{k}(L)b_{j}(x)\big|>\frac{\alpha}{2}\Big\}\Big)=:III+IV.
\end{align*}

The estimate of the term $III$ is delicate. We define $\nu_{0}\in \mathbb{N}$ such that $2^{\nu_{0}}\leq 2(1+|t|)^{p_{0}\sigma_{p_{0}}}< 2^{\nu_{0}+1}$. It is easy to see that $\nu_{0}$ is well-defined. Next we show the following lemma, which plays a crucial role in estimating the term $III$.

\begin{lemma}\label{lemm}
With the notation above, then for any $\frac{n}{2}<s<mM$, there exists a positive constant $C$ independent of $k>k_{0}$, $j\in J_{k}$ and $\nu\geq \nu_{0}$, such that
\begin{align}\label{key}
\|\chi_{C_{\nu}(B_{j})}e^{itL}F_{k}(L)\chi_{B_{j}}f\|_{2}\leq C2^{-\nu s}\mu(B_{j})^{-\sigma_{p_{0}}}(1+|t|)^{p_{0}\sigma_{p_{0}} (s+\frac{n}{2})}\|f\|_{p_{0}}.
\end{align}
\end{lemma}
\begin{proof}
Let $\phi$ be a non-negative $C_{c}^{\infty}$ function on $\mathbb{R}$ such that ${\rm supp}\phi \subseteq (1/4, 1)$ and
\begin{align*}
\sum_{\ell\in\mathbb{Z}}\phi(2^{-\ell}\lambda)=1,\ \ \forall \lambda>0,
\end{align*}
and let $\phi_{\ell}(\lambda)$ denote the function $\phi(2^{-\ell}\lambda)$.

By the spectral theory, one writes
\begin{align*}
e^{itL}F_{k}(L)=\sum_{\ell=-\infty}^{m(k-k_{0})/(m-1)}e^{itL}(I+L)^{-\sigma_{p_{0}} n}(I-e^{-2^{mk}L})^{M}\varphi_{0}(2^{-m(k-k_{0})/(m-1)}L)\phi(2^{-\ell}L).
\end{align*}
Set $F_{k,\ell}(\lambda):=e^{it\lambda}(1+\lambda)^{-\sigma_{p_{0}} n}(I-e^{-2^{mk}\lambda})^{M}\varphi_{0}(2^{-m(k-k_{0})/(m-1)}\lambda)\phi(2^{-\ell}\lambda)$, then we apply Minkowski's inequality, estimate (\ref{besov2}) and the doubling condition(\ref{doubling}) to get that
\begin{align}\label{c1}
\|\chi_{C_{\nu}(B_{j})}e^{itL}F_{k}(L)\chi_{B_{j}}\|_{p_{0}\rightarrow 2}
&\leq \sum_{\ell=-\infty}^{m(k-k_{0})/(m-1)}\|\chi_{C_{\nu}(B_{j})}F_{k,\ell}(L)\chi_{B_{j}}\|_{p_{0}\rightarrow 2}\nonumber\\
&\leq C\sum_{\ell=-\infty}^{m(k-k_{0})/(m-1)}\mu(B(x_{B_{j}},2^{-\ell/m}))^{-\sigma_{p_{0}}}(2^{\ell/m}2^{\nu} r_{B_{j}})^{-s}\|\delta_{2^{\ell}}F_{k,\ell}\|_{B^s}\nonumber\\
&\leq C2^{-\nu s}\mu(B_{j})^{-\sigma_{p_{0}}}\sum_{\ell=-\infty}^{m(k-k_{0})/(m-1)}(1+2^{\ell/m}r_{B_{j}})^{\sigma_{p_{0}} n}(2^{\ell/m}r_{B_{j}})^{-s}\|\delta_{2^{\ell}}F_{k,\ell}\|_{B^{s}}.
\end{align}

To go on, we claim that
\begin{align}\label{besov}
\|\delta_{2^{\ell}}F_{k,\ell}\|_{B^{s}}\leq C{\rm min}\{1,2^{(\ell+mk)M}\}{\rm min}\{1,2^{-\sigma_{p_{0}} n\ell}\}{\rm max}\{1,(2^{\ell}(1+|t|))^{s}\}.
\end{align}
Let us show the claim (\ref{besov}). Now we let $\eta\in C_{c}^{\infty}(\mathbb{R})$ with ${\rm supp}\eta\subset[1/8,2]$ and $\eta(\lambda)=1$ for $\lambda\in[1/4,1]$. One has
\begin{align*}
&\|\delta_{2^{\ell}}F_{k,\ell}(\lambda)\|_{B^{s}}\\
=&\|e^{it2^{\ell}\lambda}(1+2^{\ell}\lambda)^{-\sigma_{p_{0}} n}(1-e^{-2^{mk+\ell}\lambda})^{M}\varphi_{0}(2^{-m(k-k_{0})/(m-1)}2^{\ell}\lambda)\phi(\lambda)\|_{B^{s}}\\
\leq & \|\eta(\lambda)\varphi_{0}(2^{-m(k-k_{0})/(m-1)}2^{\ell}\lambda)\|_{B^{s}}\|\eta(\lambda)(1-e^{-2^{mk+\ell}\lambda})^{M}\|_{B^{s}}\|\phi(\lambda)e^{it2^{\ell}\lambda}(1+2^{\ell}\lambda)^{-\sigma_{p_{0}} n}\|_{B^{s}}\\
\leq & C \|\eta(\lambda)\varphi_{0}(2^{-m(k-k_{0})/(m-1)}2^{\ell}\lambda)\|_{C^{s+2}}\|\eta(\lambda)(1-e^{-2^{mk+\ell}\lambda})^{M}\|_{C^{s+2}}\|\phi(\lambda)e^{it2^{\ell}\lambda}(1+2^{\ell}\lambda)^{-\sigma_{p_{0}} n}\|_{B^{s}}.
\end{align*}
Note that $\ell\leq m(k-k_{0})/(m-1)$ and ${\rm supp}\eta\subset[1/8,2]$,
\begin{align*}
\|\eta(\lambda)\varphi_{0}(2^{-m(k-k_{0})/(m-1)}2^{\ell}\lambda)\|_{C^{s+2}}\leq C
\end{align*}
and
\begin{align*}
\|\eta(\lambda)(1-e^{-2^{mk+\ell}\lambda})^{M}\|_{C^{s+2}}\leq C{\rm min}\{1,2^{(\ell+mk)M}\}
\end{align*}
with $C$ independent of $k$ and $\ell$.

As for the third term $\|\phi(\lambda)e^{it2^{\ell}\lambda}(1+2^{\ell}\lambda)^{-\sigma_{p_{0}} n}\|_{B^{s}}$,we note that the Fourier transform $\mathcal{F}(\phi e^{it2^{\ell}\cdot}(1+2^{\ell}\cdot)^{-\sigma_{p_{0}} n})(\tau)$ of $\phi(\lambda)e^{it2^{\ell}\lambda}(1+2^{\ell}\lambda)^{-\sigma_{p_{0}} n}$ is given by
\begin{align*}
\mathcal{F}(\phi e^{it2^{\ell}\cdot}(1+2^{\ell}\cdot)^{-\sigma_{p_{0}} n})(\tau)
=\int_{\mathbb{R}}\phi(\lambda)\frac{e^{i(2^{\ell}t-\tau)\lambda}}{(1+2^{\ell}\lambda)^{\sigma_{p_{0}} n}}d\lambda.
\end{align*}
Integration by parts gives for every $N\in \mathbb{N}$,
\begin{align*}
|\mathcal{F}(\phi e^{it2^{\ell}\cdot}(1+2^{\ell}\cdot)^{-\sigma_{p_{0}} n})(\tau)|
\leq C(1+2^{\ell})^{-\sigma_{p_{0}} n}(1+|2^{\ell}t-\tau|)^{-N},
\end{align*}
which yields that,
\begin{align*}
\|\phi(\lambda)e^{it2^{\ell}\lambda}(1+2^{\ell}\lambda)^{-\sigma_{p_{0}} n}\|_{B^{s}}
&\leq C{\rm min}\{1,2^{- \sigma_{p_{0}} n \ell}\}\int_{\mathbb{R}}(1+|2^{\ell}t-\tau|)^{-N}(1+|\tau|)^{s}d\tau\\
&\leq C{\rm min}\{1,2^{- \sigma_{p_{0}} n \ell}\}(1+2^{\ell}t)^{s}\\
&\leq C{\rm min}\{1,2^{- \sigma_{p_{0}} n \ell}\}{\rm max}\{1,(2^{\ell}(1+|t|))^{s}\}.
\end{align*}
Hence, (\ref{besov}) holds.

It follows from (\ref{c1}) and (\ref{besov}) that
\begin{align*}
&\|\chi_{C_{\nu}(B_{j})}e^{itL}F_{k}(L)\chi_{B_{j}}\|_{p_{0}\rightarrow 2}\\
\leq& C 2^{-\nu s}\mu(B_{j})^{-\sigma_{p_{0}}}\sum_{\ell=-\infty}^{m(k-k_{0})/(m-1)}(1+2^{\ell/m}r_{B_{j}})^{\sigma_{p_{0}} n}(2^{\ell/m}r_{B_{j}})^{-s}{\rm min}\{1,2^{(\ell+mk)M}\}{\rm min}\{1,2^{-\sigma_{p_{0}} n\ell}\}{\rm max}\{1,(2^{\ell}(1+|t|))^{s}\}\\
\leq& C 2^{-\nu s}\mu(B_{j})^{-\sigma_{p_{0}}}r_{B_{j}}^{\sigma_{p_{0}} n-s}\sum_{\ell=1}^{m(k-k_{0})/(m-1)}2^{(s-\sigma_{p_{0}} n)(1-\frac{1}{m})\ell}(1+|t|)^{s}+ C_{s} 2^{-\nu s}\mu(B_{j})^{-\sigma_{p_{0}}}r_{B_{j}}^{\sigma_{p_{0}} n -s}\sum_{\ell=-2k_{0}/p_0}^{0}2^{(s+\frac{\sigma_{p_{0}} n-s}{m})\ell}(1+|t|)^{s}\\
+& C 2^{-\nu s}\mu(B_{j})^{-\sigma_{p_{0}}}r_{B_{j}}^{\sigma_{p_{0}} n-s}\sum_{\ell=-mk}^{-2k_{0}/p_0}2^{\frac{\ell}{m}(\sigma_{p_{0}} n-s)}+ C_{s} 2^{-\nu s}\mu(B_{j})^{-\sigma_{p_{0}}}\sum_{\ell=-\infty}^{-mk}2^{(\ell +mk)(M-\frac{s}{m})}\\
\leq& C 2^{-\nu s}\mu(B_{j})^{-\sigma_{p_{0}}}(1+|t|)^{p_{0}\sigma_{p_{0}}(s+\frac{n}{2})},
\end{align*}
where in the last inequality we used the fact that $\frac{n}{2}<s<mM$.

This finishes the proof of Lemma \ref{lemm}.
\end{proof}
Back to the estimate of the term $III$, we now apply Cauchy-Schwartz inequality to obtain
\begin{align}\label{RHHS}
III&=\mu\Big(\Big\{x\in\Omega_{t}^{c}:\big|\sum_{k> k_{0}}\sum_{j\in J_{k}}e^{itL}F_{k}(L)b_{j}(x)\big|>\alpha\Big\}\Big)\nonumber\\
&\leq \alpha^{-1}\sum_{k> k_{0}}\sum_{j\in J_{k}}\sum_{\nu=\nu_{0}}^{\infty}\int_{C_{\nu}(B_{j})}\big|e^{itL}F_{k}(L)b_{j}(x)\big|d\mu(x)\nonumber\\
&\leq \alpha^{-1}\sum_{k> k_{0}}\sum_{j\in J_{k}}\sum_{\nu=\nu_{0}}^{\infty}\mu\big(C_{\nu}(B_{j})\big)^{\frac{1}{2}}\bigg(\int_{C_{\nu}(B_{j})}\big|e^{itL}F_{k}(L)b_{j}(x)\big|^{2}d\mu(x)\bigg)^{\frac{1}{2}}.
\end{align}
This, in combination with the doubling condition (\ref{doubling}) and Lemma \ref{lemm}, we conclude that
\begin{align*}
{\rm RHS\ of\ \eqref{RHHS}}
&\leq C\alpha^{-1}\sum_{k> k_{0}}\sum_{j\in J_{k}}\sum_{\nu=\nu_{0}}^{\infty}2^{\frac{\nu n}{2}}\mu(B_{j})^{\frac{1}{2}}\|\chi_{C_{\nu}(B_{j})}e^{itL}F_{k}(L)\chi_{B_{j}}\|_{p_{0}\rightarrow 2}\|b_{j}\|_{p_{0}}\\
&\leq C\sum_{k> k_{0}}\sum_{j\in J_{k}}\sum_{\nu=\nu_{0}}^{\infty}\mu(B_{j})2^{\nu(\frac{n}{2}-s)}(1+|t|)^{p_{0}\sigma_{p_{0}}(s+\frac{n}{2})}\\
&\leq C(1+|t|)^{p_{0}\sigma_{p_{0}} n}\sum_{j}\mu(B_{j})\\
&\leq C(1+|t|)^{p_{0}\sigma_{p_{0}} n}\alpha^{-p_{0}}\|f\|_{p_{0}}^{p_{0}}.
\end{align*}

Concerning the term $IV$, since the Schr\"{o}dinger group $e^{-itL}$ is bounded on $L^{2}(X)$, we have
\begin{align*}
IV
&\leq C\alpha^{-2}\Big\|\sum_{k>k_{0}}\sum_{j\in J_{k}}e^{itL}G_{k}(L)b_{j}\Big\|_{2}^{2}\\
&\leq C\alpha^{-2}\Big\|\sum_{k>k_{0}}\sum_{j\in J_{k}}G_{k}(L)b_{j}\Big\|_{2}^{2}\\
&\leq C\alpha^{-2}\Big\|\sum_{k>k_{0}}\sum_{j\in J_{k}}\chi_{2B_{j}}G_{k}(L)b_{j}\Big\|_{2}^{2}+C\alpha^{-2}\Big\|\sum_{k>k_{0}}\sum_{j\in J_{k}}\chi_{(2B_{j})^{c}}G_{k}(L)b_{j}\Big\|_{2}^{2}=:IV_{1}+IV_{2}.
\end{align*}

To handle the term $IV_{1}$, we first note that
\begin{align*}
&\|G_{k}(L)b_{j}\|_{2}\leq \|G_{k}(L)(I+2^{-m(k-k_{0})/(m-1)}L)^{\sigma_{p_{0}} n}\|_{2\rightarrow 2}\|(I+2^{-m(k-k_{0})/(m-1)}L)^{-\sigma_{p_{0}} n}b_j\|_{2}\\
&\leq \|(1+\lambda)^{-\sigma_{p_{0}} n}(1-e^{-2^{mk}\lambda})^{M}\varphi_{1}(2^{-m(k-k_{0})/(m-1)}\lambda)(1+2^{-m(k-k_{0})/(m-1)}\lambda)^{\sigma_{p_{0}} n}\|_{\infty}\|(I+2^{-m(k-k_{0})/(m-1)}L)^{-\sigma_{p_{0}} n}b_j\|_{2}\\
&\leq C2^{-\frac{m(k-k_{0})}{m-1}\sigma_{p_{0}} n}\|(I+2^{-m(k-k_{0})/(m-1)}L)^{-\sigma_{p_{0}} n}b_j\|_{2}.
\end{align*}

Since the $2B_{j}$'s have bounded overlaps, we apply Minkowski's inequality to obtain
\begin{align}\label{b4}
\Big\|\sum_{k>k_{0}}\sum_{j\in J_{k}}\chi_{2B_{j}}G_{k}(L)b_{j}\Big\|_{2}^{2}
&\leq C\sum_{k>k_{0}}\sum_{j\in J_{k}}\int_{X}|G_{k}(L)b_{j}(x)|^{2}d\mu(x)\nonumber\\
&\leq C\sum_{k>k_{0}}\sum_{j\in J_{k}}2^{-\frac{2m(k-k_{0})}{m-1}\sigma_{p_{0}} n}\|(I+2^{-m(k-k_{0})/(m-1)}L)^{-\sigma_{p_{0}} n}b_j\|_{2}^{2}\nonumber\\
&\leq C\sum_{k>k_{0}}\sum_{j\in J_{k}}2^{-\frac{2m(k-k_{0})}{m-1}\sigma_{p_{0}} n}\|\chi_{2B_{j}}(I+2^{-m(k-k_{0})/(m-1)}L)^{-\sigma_{p_{0}} n}\chi_{2B_{j}}\|_{p_{0}\rightarrow 2}^{2}\|b_{j}\|_{p_{0}}^{2}\nonumber\\
&+C\sum_{k>k_{0}}\sum_{j\in J_{k}}\sum_{\nu=2}^{\infty}2^{-\frac{2m(k-k_{0})}{m-1}\sigma_{p_{0}} n}\|\chi_{C_{\nu}(B_{j})}(I+2^{-m(k-k_{0})/(m-1)}L)^{-\sigma_{p_{0}} n}\chi_{B_{j}}\|_{p_{0}\rightarrow 2}^{2}\|b_{j}\|_{p_{0}}^{2}.
\end{align}
By applying the representation formula
\begin{align*}
(I+2^{-m(k-k_{0})/(m-1)}L)^{-\sigma_{p_{0}} n}=\frac{1}{\Gamma(\sigma_{p_{0}} n)}\int_{0}^{\infty}e^{-\lambda 2^{-m(k-k_{0})/(m-1)}L}e^{-\lambda}\lambda^{\sigma_{p_{0}} n-1}d\lambda
\end{align*}
and Lemma \ref{real}, we conclude that
\begin{align*}
\|\chi_{2B_{j}}(I+2^{-m(k-k_{0})/(m-1)}L)^{-\sigma_{p_{0}} n}\chi_{2B_{j}}\|_{p_{0}\rightarrow 2}
&\leq C\int_{0}^{\infty}\|\chi_{2B_{j}}e^{-\lambda 2^{-m(k-k_{0})/(m-1)}L}\chi_{2B_{j}}\|_{p_{0}\rightarrow 2}e^{-\lambda}\lambda^{\sigma_{p_{0}} n-1}d\lambda\\
&\leq C\int_{0}^{\infty}\mu(B(x_{B_{j}},\lambda^{\frac{1}{m}}2^{-\frac{k-k_{0}}{m-1}}))^{-\sigma_{p_{0}}}e^{-\lambda}\lambda^{\sigma_{p_{0}} n-1}d\lambda
\end{align*}
and that for any $\nu\geq 2$,
\begin{align*}
\|\chi_{C_{\nu}(B_{j})}(I+2^{-m(k-k_{0})/(m-1)}L)^{-\sigma_{p_{0}} n}\chi_{B_{j}}\|_{p_{0}\rightarrow 2}
&\leq C\int_{0}^{\infty}\|\chi_{C_{\nu}(B_{j})}e^{-\lambda 2^{-m(k-k_{0})/(m-1)}L}\chi_{B_{j}}\|_{p_{0}\rightarrow 2}e^{-\lambda}\lambda^{\sigma_{p_{0}} n-1}d\lambda\\
\leq C\int_{0}^{\infty}&\mu(B(x_{B_{j}},\lambda^{\frac{1}{m}}2^{-\frac{k-k_{0}}{m-1}}))^{-\sigma_{p_{0}}}{\rm exp}\bigg(-c\Big(\frac{ 2^{\frac{k-k_{0}}{m-1}}2^\nu r_{B_{j}}}{\lambda^{1/m}}\Big)^{\frac{m}{m-1}}\bigg)e^{-\lambda}\lambda^{\sigma_{p_{0}} n-1}d\lambda.
\end{align*}

The doubling condition (\ref{doubling}) implies that
\begin{align*}
\frac{1}{\mu(B(x_{B_{j}},\lambda^{\frac{1}{m}}2^{-\frac{k-k_{0}}{m-1}}))}
&=\frac{\mu(B(x_{B_{j}},2^{-\frac{k-k_{0}}{m-1}}))}{\mu(B(x_{B_{j}},\lambda^{\frac{1}{m}}2^{-\frac{k-k_{0}}{m-1}}))}\cdot \frac{\mu(B(x_{B_{j}},2^{k}))}{\mu(B(x_{B_{j}},2^{-\frac{k-k_{0}}{m-1}}))}\cdot \frac{1}{\mu(B(x_{B_{j}},2^{k}))}\\
&\leq C2^{\frac{m(k-k_{0})}{m-1}n}(1+|t|)^{\frac{p_{0}n}{2}}\Big(1+\frac{1}{\lambda^{1/m}}\Big)^{n}\frac{1}{\mu(B_{j})}.
\end{align*}
This indicates that
\begin{align}\label{b2}
\|\chi_{2B_{j}}(I+2^{-\frac{m(k-k_{0})}{m-1}}L)^{-\sigma_{p_{0}} n}\chi_{2B_{j}}\|_{p_{0}\rightarrow 2}
\leq C2^{\frac{m(k-k_{0})}{m-1}\sigma_{p_{0}} n}(1+|t|)^{\frac{p_{0}\sigma_{p_{0}} n}{2}}\mu(B_{j})^{-\sigma_{p_{0}}},
\end{align}
and that
\begin{align}\label{b3}
\|\chi_{C_{\nu}(B_{j})}(I+2^{-\frac{m(k-k_{0})}{m-1}}L)^{-\sigma_{p_{0}} n}\chi_{B_{j}}\|_{p_{0}\rightarrow 2}
\leq& C2^{\frac{m(k-k_{0})}{m-1}\sigma_{p_{0}} n}(1+|t|)^{\frac{p_{0}\sigma_{p_{0}} n}{2}}\mu(B_{j})^{-\sigma_{p_{0}}}\nonumber\\
\times&\int_{0}^{\infty}{\rm exp}\bigg(-c\Big(\frac{2^{\frac{k-k_{0}}{m-1}}2^\nu r_{B_{j}}}{\lambda^{1/m}}\Big)^{\frac{m}{m-1}}\bigg)e^{-\lambda}(1+\frac{1}{\lambda^{1/m}})^{\sigma_{p_{0}} n}\lambda^{\sigma_{p_{0}} n-1}d\lambda\nonumber\\
\leq& C2^{\frac{m(k-k_{0})}{m-1}\sigma_{p_{0}} n}(1+|t|)^{\frac{p_{0}\sigma_{p_{0}} n}{2}}\mu(B_{j})^{-\sigma_{p_{0}}}\Big(2^{\frac{k-k_{0}}{m-1}}2^\nu r_{B_{j}}\Big)^{-N}.
\end{align}

Observing that $2^{\frac{k-k_{0}}{m-1}}r_{B_{j}}\geq 1$ for $k>k_{0}$, we combine (\ref{b4}),(\ref{b2}) with (\ref{b3}) to conclude that
\begin{align*}
\Big\|\sum_{k>k_{0}}\sum_{j\in J_{k}}\chi_{2B_{j}}G_{k}(L)b_{j}\Big\|_{2}^{2}
&\leq C(1+|t|)^{p_{0}\sigma_{p_{0}} n}\sum_{k>k_{0}}\sum_{j\in J_{k}}\mu(B_{j})^{-2\sigma_{p_{0}}}\|b_{j}\|_{p_{0}}^{2}\\
&+ C(1+|t|)^{p_{0}\sigma_{p_{0}} n}\sum_{k>k_{0}}\sum_{j\in J_{k}}\sum_{\nu=2}^\infty\mu(B_{j})^{-2\sigma_{p_{0}}}\|b_{j}\|_{p_{0}}^{2}\Big(2^{\frac{k-k_{0}}{m-1}}2^\nu r_{B_{j}}\Big)^{-N}\\
&\leq C(1+|t|)^{p_{0}\sigma_{p_{0}} n}\sum_{k>k_{0}}\sum_{j\in J_{k}}\mu(B_{j})^{-2\sigma_{p_{0}}}\|b_{j}\|_{p_{0}}^{2}\\
&\leq C(1+|t|)^{p_{0}\sigma_{p_{0}} n}\sum_{k>k_{0}}\sum_{j\in J_{k}}\mu(B_{j})^{-2\sigma_{p_{0}}}\|b_{j}\|_{p_{0}}^{2-p_0}\|b_{j}\|_{p_{0}}^{p_0}\\
&\leq C(1+|t|)^{p_{0}\sigma_{p_{0}} n}\alpha^{2-p_{0}}\sum_{j}\int_{X}|b_{j}(y)|^{p_{0}}d\mu(y)\\
&\leq C(1+|t|)^{p_{0}\sigma_{p_{0}} n}\alpha^{2-p_{0}}\|f\|_{p_{0}}^{p_{0}}.
\end{align*}

Finally, it remains to show that
\begin{align}\label{remain}
\Big\|\sum_{k>k_{0}}\sum_{j\in J_{k}}\chi_{(2B_{j})^{c}}G_{k}(L)b_{j}\Big\|_{2}^{2}\leq C\alpha^{2-p_{0}}\|f\|_{p_{0}}^{p_{0}}.
\end{align}
To continue, we claim that for any $N>\sigma_{p_{0}} n$, there exists a positive constant $C$ such that for any $\nu\geq 2$ and $k>k_{0}$, we have
\begin{align}\label{goal}
\|\chi_{C_{\nu}(B_{j})}G_{k}(L)\chi_{B_{j}}\|_{p_{0}\rightarrow 2}\leq C2^{-\nu N}\mu(B_{j})^{-\sigma_{p_{0}}}.
\end{align}

To show (\ref{goal}), we apply (\ref{besov2}) to obtain that
\begin{align}\label{a7}
\|\chi_{C_{\nu}(B_{j})}G_{k}(L)\chi_{B_{j}}\|_{p_{0}\rightarrow 2}
&\leq \sum_{\ell=-1}^{\infty}\|\chi_{C_{\nu}(B_{j})}G_{k,\ell}(L)\chi_{B_{j}}\|_{p_{0}\rightarrow 2}\nonumber\\
&\leq C \sum_{\ell=-1}^{\infty}\mu(B(x_{B_{j}},2^{-\ell/m}))^{-\sigma_{p_{0}}}(2^{\ell/m}2^\nu r_{B_{j}})^{-N}\|\delta_{2^{\ell}}G_{k,\ell}\|_{B^{N}},
\end{align}
where $G_{k,\ell}(\lambda):=G_{k}(\lambda)\phi(2^{-\ell}\lambda)$ satisfies
\begin{align*}
\|\delta_{2^{\ell}}G_{k,\ell}\|_{B^{N}}
=&\|(1+2^{\ell}\lambda)^{-\sigma_{p_{0}} n}(1-e^{-2^{mk+\ell}\lambda})^{M}\varphi_{1}(2^{-\frac{m(k-k_{0})}{m-1}}2^{\ell}\lambda)\phi(\lambda)\|_{B^{N}}\\
\leq&\|\phi(\lambda)(1+2^{\ell}\lambda)^{-\sigma_{p_{0}} n}\|_{B^{N}}\|\eta(\lambda)(1-e^{-2^{mk+\ell}\lambda})^{M}\|_{B^{N}}\|\eta(\lambda)\varphi_{1}(2^{-\frac{m(k-k_{0})}{m-1}}2^{\ell}\lambda)\|_{B^{N}}\\
\leq& C\|\phi(\lambda)(1+2^{\ell}\lambda)^{-\sigma_{p_{0}} n}\|_{C^{N+2}}\|\eta(\lambda)(1-e^{-2^{mk+\ell}\lambda})^{M}\|_{C^{N+2}}\|\eta(\lambda)\varphi_{1}(2^{-\frac{m(k-k_{0})}{m-1}}2^{\ell}\lambda)\|_{C^{N+2}}\\
\leq& C2^{-\ell \sigma_{p_{0}} n}.
\end{align*}
This, in combination with (\ref{a7}) and the doubling condition (\ref{doubling}), yields
\begin{align*}
\|\chi_{C_{\nu}(B_{j})}G_{k}(L)\chi_{B_{j}}\|_{p_{0}\rightarrow 2}
&\leq C\sum_{\ell=-1}^{\infty}\mu(B(x_{B_{j}},2^{-\ell/m}))^{-\sigma_{p_{0}}}(2^{\ell/m}2^\nu r_{B_{j}})^{-N}2^{-\ell \sigma_{p_{0}} n}\\
&\leq C2^{-\nu N}\mu(B_{j})^{-\sigma_{p_{0}}}\sum_{\ell=-1}^{\infty}(2^{\ell/m}r_{B_{j}})^{\sigma_{p_{0}} n-N}2^{-\ell\sigma_{p_{0}} n}\\
&\leq C2^{-\nu N}\mu(B_{j})^{-\sigma_{p_{0}}},
\end{align*}
where in the last inequality we used the fact that $r_{B_{j}}\geq 1$ when $k\geq k_0$. This ends the proof of estimate (\ref{goal}).

Now we follow the similar procedure as estimating the term $I$ to show (\ref{remain}).

\begin{align}\label{duiou1}
\Big\|\sum_{k>k_{0}}\sum_{j\in J_{k}}\chi_{(2B_{j})^{c}}G_{k}(L)b_{j}\Big\|_{2}^{2}
&=\sup\limits_{\|u\|_{2}=1}\bigg{(}\int_{X}u(x)\sum_{k\leq k_{0}}\sum_{j\in J_{k}}\chi_{(2B_{j})^{c}}G_{k}(L)b_{j}(x)d\mu(x)\bigg{)}^{2}\nonumber\\
&\leq \sup\limits_{\|u\|_{2}=1}\bigg{(}\sum_{k\leq k_{0}}\sum_{j\in J_{k}}\sum_{\nu=2}^{\infty}B_{\nu j}\bigg{)}^{2},
\end{align}
where $B_{\nu j}:=\int_{C_{\nu}(B_{j})}|G_{k}(L)b_{j}(x)||u(x)|d\mu(x)$.

By (\ref{max}) and (\ref{goal}), one obtains
\begin{align*}
B_{\nu j}
\leq \|G_{k}(L)b_{j}\|_{L^{2}(C_{\nu}(B_{j}))}\|u\|_{L^{2}(C_{\nu}(B_{j}))}
\leq C2^{-(N-\frac{n}{2})\nu}\alpha \mu(B_{j})\mathscr {M}(|u|^{2})(y)^{\frac{1}{2}}.
\end{align*}
Averaging over $B_{j}$ yields
\begin{align*}
B_{\nu j}
&\leq C2^{-(N-\frac{n}{2})\nu}\alpha \int_{B_{j}}\mathscr {M}(|u|^{2})(y)^{\frac{1}{2}}d\mu(y).
\end{align*}

Choosing $N$ sufficient large and then summing over $\nu\geq 1$ and $j$, we have
\begin{align*}
\Big\|\sum_{k>k_{0}}\sum_{j\in J_{k}}\chi_{(2B_{j})^{c}}G_{k}(L)b_{j}\Big\|_{2}^{2}
&\leq C\alpha^{2}\sup\limits_{\|u\|_{2}=1}\bigg(\int_{\mathop{\cup}\limits_{j}B_{j}}\mathscr {M}(|u|^{2})(y)^{\frac{1}{2}}d\mu(y)\bigg)^{2}\\
&\leq C\alpha^{2}\sup\limits_{\|u\|_{2}=1}\mu\Big(\mathop{\cup}\limits_{j}B_{j}\Big)\|u^{2}\|_{1}\\
&\leq C\alpha^{2-p_{0}}\|f\|_{p_{0}}^{p_{0}}.
\end{align*}
The proof of Theorem \ref{main} is complete.

\bigskip

 \noindent
 {\bf Acknowledgements}: The author would like to thank P. Chen and L.X.  Yan for helpful discussions. Z.J. Fan was supported by International Program for Ph.D. Candidates, Sun Yat-Sen University.


\begin{thebibliography}{99999}

\bibitem  {A} G. Alexopoulos, Oscillating multipliers on Lie groups and Riemannian manifolds. {\it Tohoku Math. J.}
  {\bf 46} (1994), 457-468.


\bibitem {Aus} P. Auscher, on necessary and sufficient conditions for $L^{p}$-estimates of Riesz transforms associated to elliptic operators on $\mathbb{R}^{n}$ and related estimates. {\it Mem. Amer. Math. Soc}
{\bf 186} (2007), no 871.

\bibitem{Blunck0}  S. Blunck and P.C. Kunstmann,  Calder\'on-Zygmund theory for non-integral operators and the $H^\infty$ functional calculus.
{\it Rev. Mat. Iberoamericana} {\bf 19}   (2003),     919--942.

\bibitem{Blunck00}
S. Blunck, A H\"ormander-type spectral multiplier theorem for operators without heat kernel.
{\it Ann. Sc. Norm. Super. Pisa Cl. } Sci. (5),   (2003),   449--459.

\bibitem {Blunck1} S. Blunck and P.C. Kunstmann, Generalized Gaussian estimates and the Legendre transform. {\it J. Operator Theory }
{\bf 53}  (2),(2005),351-365.

\bibitem {Blunck2} S. Blunck, Generalized Gaussian estimates and Riesz means of Schr\"{o}dinger groups. {\it J. Aust. Math. Soc }.
{\bf 82} (2007),149-162.



\bibitem{Br} P. Brenner, The Cauchy problem for systems in $L_p$ and $L_{p, \alpha}$.
 {\it Ark. Mat.} {\bf 2} (1973), 75-101.


\bibitem{BDN} T.A. Bui, P. D'Ancona, F. Nicola,
 Sharp $L^p$ estimates for Schr\"odinger groups on spaces of homogeneous type. {\it Rev. Mat. Iberoam.}
{\bf 36} (2020), 455--484.

\bibitem {CW} R. Coifman and G. Weiss, Analyse harmonique non-commutative sur certains espaces homog\`{e}nes. {\it Lecture Notes in Math }.
{\bf 242} . Springer, Berlin-New York, (1971).

\bibitem{CCO} G. Carron, T.  Coulhon and E.M. Ouhabaz,  Gaussian estimates and $L^p$-boundedness of Riesz means.
{\it J. Evol. Equ.} {\bf 2} (2002), 299--317.


\bibitem{CKS} S. Chanillo, D.S. Kurtz and G. Sampson, Weighted weak $(1,1)$ and weighted $L^p$ estimates for oscillating kernels.
{\it Trans. Amer. Math. Soc.} {\bf 295} (1986), 127--145.





\bibitem{CDLY} P. Chen, X.T. Duong, J. Li and L.X. Yan, Sharp endpoint $L^p$ estimates for Schr\"odinger groups.
To appear in Math. Ann.

\bibitem {hardy} P. Chen, X.T. Duong, J. Li and L.X. Yan, Sharp endpoint estiamtes for Schr\"{o}dinger groups on Hardy spaces. Available at arXiv: 1902.08875 (2019).

\bibitem {CDLSY} P. Chen, X.T.Duong, J. Li, L. Song and L.X. Yan, Weak type $(1,1)$ bounds for Schr\"{o}dinger groups. Available at arXiv: 1906.05519 (2019).


\bibitem{DN} P. D' Ancona and F. Nicola,  Sharp $L^p$ estimates for Schr\"odinger groups.
  {\it  Rev. Mat. Iberoam.}  {\bf 32} (2016),1019--1038.


  \bibitem{D2} E.B. Davies, Limits on $L^p$ regularity of self-adjoint elliptic operators.
  {\it J. Diff. Equa.} {\bf 135}(1997), 83--102.

\bibitem {Davies} E.B. Davies, Uniformly elliptic operators with measurable coefficients. {\it J.Funct.Anal }.
{\bf 132} (1995),141-169.





\bibitem{DM} X.T. Duong and A. McIntosh, Singular integral operators with
non-smooth kernels on irregular domains. {\it Rev. Mat. Iberoamericana} {\bf 15}
(1999), no. 2, 233--265.



 \bibitem{DR} X.T. Duong and D.W. Robinson,  Semigroup kernels,
Poisson bounds, and holomorphic functional calculus. {\it J. Funct. Anal.}
{\bf 142} (1996), no. 1, 89--128.

\bibitem{FS} C. Fefferman and E.M. Stein, {\it $H^p$ spaces of
 several variables},   Acta
Math., {\bf 129} (1972), 137--195.


   \bibitem{EK} O. El-Mennaoui and V. Keyantuo,  On the Schr\"odinger equation in $L^p$ spaces. {\it Math. Ann}. {\bf 304} (1996), 293-302.


\bibitem{H} M. Hieber, Integrated semigroups and differential operators on $L^p$ spaces. {\it Math. Ann}. {\bf 291} (1991), 1-16.

 \bibitem {H1} L. H\"ormander,  Estimates for translation invariant
operators in $L^p$ spaces.   {\it Acta Math.}  {\bf 104} (1960), 93--140.



 \bibitem{JN} A. Jensen and S.  Nakamura,  Mapping properties of functions of Schr\"odinger operators between
$L^p$-spaces and Besov spaces. In {\it Spectral and scattering theory and applications},
187--209. Adv. Stud. Pure Math. {\bf 23}, Math. Soc. Japan, Tokyo, 1994.


\bibitem{JN2} A. Jensen and S.  Nakamura, $L^p$-mapping properties of functions of Schr\"odinger operators
 and their applications to scattering theory. {\it J. Math. Soc. Japan} {\bf 47} (1995),  253--273.


  \bibitem{KU} P. Kunstmann and M. Uhl,   Spectral multiplier theorems of
  H\"ormander type on Hardy and Lebesgue spaces. {\it J. Operator Theory} {\bf 73 }
  (2015),  27--69.

 \bibitem{La} E. Lanconelli, Valutazioni in $L^p(\RN)$ della soluzione del problema di Cauchy per l'equazione di Schr\"odinger.
 {\it Boll. Un. Mat. Ital. (4)} {\bf 1} (1968), 591--607.


 \bibitem{Lo} N. Lohou\'e, Estimations des sommes de Riesz d'op\'erateurs de Schr\"odinger sur les vari\'et\'es
 riemanniennes et les groupes de Lie. {\it C.R.A.S. Paris}. {\bf 315} (1992), 13-18.


  \bibitem {LSV} V. Liskevich, Z. Sobol and H. Vogt, On the $L^p$ theory
  of $C^0$-semigroups associated with second-order elliptic operators {\rm II}.
  {\it J. Funct. Anal.} {\bf 193} (2002),   55--76.

 \bibitem  {Mc} A. McIntosh,  Operators which have an $H_\infty$ functional
calculus, {\it  Miniconference on operator theory and partial differential
equations (North Ryde, 1986)}, 210-231, Proceedings of the Centre  for
Mathematical Analysis, Australian National University, {\bf 14}.
 Australian National University, Canberra, 1986.

 \bibitem{Mi1} A. Miyachi, On some Fourier multipliers for $H^p(R^n)$. {\it J. Fac. Sci. Univ. Tokyo Sect IA Math}
 {\bf  27} (1980), 157-179.


\bibitem {O} E.M. Ouhabaz, {\it Analysis of Heat Equations on
Domains}.  London Mathematical Society Monographs Series, {\bf 31}.
Princeton University Press, Princeton, NJ, 2005.

\bibitem{ScV} G. Schreieck and J. Voigt, Stability of the $L_p$-spectrum of generalized Schr\"odinger
 operators with form small negative part of the potential. In {\it Function Analysis (Essen, 1991)}, 95-105.
 Lecture Notes in Pure and Appl. Math., {\bf 150}. Dekker, New York, 1994.



\bibitem{Sj} S. Sj\"ostrand, On the Riesz means of the solutions of the Schr\"odinger equation. {\it Ann. Scuola Norm. Sup. Pisa.}
{\bf 24} (1970), 331-348.






\end{thebibliography}
\end{document}